\documentclass[12pt,reqno]{amsart}
\usepackage{amssymb,amsmath,amsthm,enumerate,verbatim,bbm,bm, mathrsfs, mathtools, mathrsfs}

\usepackage[a4paper]{geometry}

\DeclareMathOperator{\supp}{supp}

\renewcommand\Im{\hbox{{\rm Im}}\,}

\newcommand{\abs}[1]{\lvert#1\rvert}
\newcommand{\Abs}[1]{\left\lvert#1\right\rvert}
\newcommand{\norm}[1]{\lVert#1\rVert}

\newcommand{\Norm}[1]{\left\lVert#1\right\rVert}

\newcommand{\floor}[1]{\lfloor#1\rfloor}

\newcommand{\bbT}{{\mathbb T}}
\newcommand{\bbR}{{\mathbb R}}
\newcommand{\bbC}{{\mathbb C}}

\newcommand{\bbN}{{\mathbb N}}
\newcommand{\bbZ}{{\mathbb Z}}

\newcommand{\bbD}{{\mathbb D}}

\newcommand{\wh}{\widehat}
\newcommand{\wc}{\widecheck}
\newcommand{\ba}{{\mathbf a}}

\newcommand{\bb}{{\mathbf b}}

\newcommand{\bH}{{\mathbf H}}
\newcommand{\bw}{{\mathbf w}}
\newcommand{\bu}{{\mathbf u}}

\newcommand{\f}{{\mathbf f}}

\newcommand{\bpsi}{{\bm{\psi}}}

\newcommand{\bvphi}{{\bm{\varphi}}}

\newcommand{\calF}{\mathcal{F}}

\newcommand{\calP}{\mathcal{P}}
\newcommand{\calB}{\mathcal{B}}

\newcommand{\calR}{\mathcal{R}}
\newcommand{\calE}{\mathcal{E}}

\newcommand{\Sch}{\mathbf{S}}

\DeclareFontFamily{U}{mathx}{\hyphenchar\font45}
\DeclareFontShape{U}{mathx}{m}{n}{<5> <6> <7> <8> <9> <10>
<10.95> <12> <14.4> <17.28> <20.74> <24.88> mathx10}{}
\DeclareSymbolFont{mathx}{U}{mathx}{m}{n}
\DeclareFontSubstitution{U}{mathx}{m}{n}
\DeclareMathAccent{\widecheck}{0}{mathx}{"71}

\numberwithin{equation}{section}

\theoremstyle{plain}
\newtheorem{theorem}{\bf Theorem}[section]

\newtheorem{lemma}[theorem]{\bf Lemma}
\newtheorem{proposition}[theorem]{\bf Proposition}

\theoremstyle{definition}

\theoremstyle{remark}
\newtheorem*{remark*}{\bf Remark}
\newtheorem{remark}[theorem]{\bf Remark}
\newtheorem{example}[theorem]{\bf Example}

\newcommand{\wt}{\widetilde}

\newcommand{\loc}{\mathrm{loc}}

\newcommand{\1}{\mathbbm{1}}

\newcommand{\bh}{{\mathbf{h}}}
\newcommand{\bg}{{\mathbf{g}}}

\renewcommand{\[}{\begin{equation}}
\renewcommand{\]}{\end{equation}}

\begin{document} 
\title{Restriction theorems for Hankel operators}
\date{1 October 2018} 

\author{Nazar Miheisi \and   Alexander Pushnitski}
\address{Department of Mathematics,
King's College London,
Strand, London WC2R 2LS,
United Kingdom}
\email{nazar.miheisi@kcl.ac.uk}
\email{alexander.pushnitski@kcl.ac.uk}



\begin{abstract}
We consider a class of maps from integral Hankel operators
to Hankel matrices, which we call restriction maps. 
In the simplest case, such a map is simply a restriction of the integral kernel
onto integers. 
More generally, it is given by an averaging of the kernel with a sufficiently
regular weight function. 
We study the boundedness of restriction maps with respect to the operator norm and 
the Schatten norms. 
\end{abstract}

\maketitle

\section{Introduction}

\subsection{Hankel operators}
Let $\alpha = \{\alpha(j)\}_{j\ge 0}$
be a sequence of complex numbers. The \emph{Hankel matrix} $H(\alpha)$
is the ``infinite matrix'' $\{\alpha(j+k)\}_{j,k\ge 0}$, considered as a linear operator on $\ell^2(\bbZ_+)$, 
$\bbZ_+=\{0,1,2,\dots\}$, so that
$$
(H(\alpha)x)(k) = \sum_{j\ge 0} \alpha(j+k)x(j), \quad k\geq0,\quad
x=\{x(j)\}_{j\geq0}\in \ell^2(\bbZ_+).
$$
Similarly, for a \emph{kernel function} $\ba\in L^1_\loc(0,\infty)$, the \emph{integral Hankel operator}
on $L^2(0,\infty)$ is defined by the formula
$$
(\bH(\ba)\f)(t) = \int_0^\infty \ba(t+s)\f(s) \,\,ds, \quad t>0,\quad \f\in L^2(0,\infty). 
$$
In order to distinguish between these two classes of operators, we 
use boldface font for objects associated with integral Hankel operators. 

For general background on Hankel operators, see \cite{Nikolski,Peller}. 
In what follows, we will only consider
bounded Hankel matrices and bounded integral Hankel operators.

\subsection{Restrictions}

The purpose of this paper is to examine the linear map, which we call 
\emph{the restriction map}, 
between the set of integral Hankel operators and the set of Hankel matrices. 
To set the scene, let us consider the \emph{pointwise restriction} of integral kernels to integers. 
For a given kernel function $\ba$, define the sequence
\[
\alpha(j):=\ba(j+1), \quad j\geq0.
\label{a1}
\]
Of course, for this operation to make sense, the kernel function $\ba$ has to be 
continuous. Here is our first result; we denote by $\Sch_p$, $0<p<\infty$, the 
standard Schatten class of compact operators (see Section~\ref{sec.b}). 

\begin{theorem}\label{thm.a1}
Let $\bH(\ba)\in\Sch_p$ for some $0<p\leq1$. Then the kernel function $\ba(t)$  
is continuous in $t>0$, so the restriction \eqref{a1} is well defined.
The operator $H(\alpha)$ is in $\Sch_p$ and we have the estimate
\[
\norm{H(\alpha)}_{\Sch_p}\leq C_p\norm{\bH(\ba)}_{\Sch_p}. 
\label{a2}
\]
\end{theorem}
The continuity of the kernel function $\ba$ for trace class integral Hankel operators
is well known (see e.g. \cite[Corollary 7.10]{Partington}); the main point here
is the estimate \eqref{a2}. 
In Section~\ref{sec.c} we give a slightly more precise version of Theorem~\ref{thm.a1} and 
show that it does not extend to $p>1$. Further, we show that if 
we restrict the map $\bH(\ba)\mapsto H(\alpha)$ to \emph{non-negative} integral 
Hankel operators, then it is bounded in $\Sch_p$ norm for all $0<p<\infty$
(and also in the operator norm). 

Further, along with the pointwise restriction \eqref{a1}, we consider 
the following \emph{restrictions by averaging}. 
For a suitably regular function $\bvphi$ on $\bbR$ and for 
a kernel function $\ba$, 
we define the restriction $\calR_\bvphi\ba$  to be the sequence
$$
\calR_\bvphi \ba(j)=\int_0^\infty \ba(t)\bvphi(t-j)\,\,dt, \quad  j\geq0.
$$
In particular, formally taking $\bvphi(t)=\delta(t-1)$, where $\delta$ is the 
Dirac $\delta$-function, we recover the pointwise restrictions \eqref{a1}.  
In Section~\ref{sec.d} we prove that, under suitable regularity conditions on $\bvphi$, 
the map 
$$
\bH(\ba)\mapsto H(\calR_\bvphi\ba)
$$
is bounded in $\Sch_p$ norm for all $0<p<\infty$
(and also in the operator norm). 
We also relate this result to the well known unitary equivalence between Hankel matrices
and integral Hankel operators.

This paper appeared as an attempt to consider one of the technical ingredients
of \cite{MP} on a more systematic basis. 
Theorem~\ref{thm.a1} and its proof is based on the same set of ideas as \cite[Theorem 3.2]{MP}. 

The results of this paper seem to parallel some restriction theorems
for Fourier multipliers; see e.g. \cite{de Leeuw,Igari,Coifman-weiss}. 
However, this connection is not completely understood (at least by the authors).

We note in passing that one can consider a converse operation, 
an \emph{extension} of a Hankel matrix to an integral kernel. 
For a suitably regular function $\bvphi$ and 
a sequence $\alpha=\{\alpha(j)\}_{j\geq0}$, one can define the 
\emph{extension} $\calE_{\bvphi}\alpha$ to be the function
$$
\calE_{\bvphi}\alpha(t)=\sum_{k\geq0}\alpha(k)\bvphi(t-k), \quad t>0,
$$
and one can consider the map
$$
H(\alpha)\mapsto \bH(\calE_{\bvphi}\alpha). 
$$
Although some Schatten norm boundedness results for this map are not
difficult to prove, we have not succeeded in finding a coherent set of estimates
for it and therefore we do not discuss extensions here.

\subsection{Symbols}
For a bounded Hankel matrix $H(\alpha)$, its analytic symbol is the function 
$$
\wc \alpha(z)=\sum_{m\geq0} \alpha_m z^m, 
\quad \abs{z}<1.
$$
Similarly, for a bounded integral Hankel operator $\bH(\ba)$, its analytic symbol
is the function 
$$
\wc\ba(\xi)=\int_0^\infty \ba(t)e^{2\pi i t\xi}dt, \quad \Im\xi>0.
$$
It is instructive to view restriction maps on Hankel operators
in terms of the symbols. 
If $\alpha=\calR_\bvphi\ba$, then for the symbols we have
\begin{equation}
\wc\alpha(z)=\int_\bbR \frac{\wc\ba(\xi+i0)\wc\bvphi(-\xi+i0)}{1-ze^{-2\pi i\xi}}d\xi, \quad \abs{z}<1.
\label{r}
\end{equation}
In particular, for the pointwise restriction \eqref{a1} we have
\[
\wc\alpha(e^{2\pi i\xi})=e^{-2\pi i\xi}\sum_{j\in\bbZ}\wc \ba(\xi-j), \quad \Im \xi>0.
\label{r1}
\]
Since Schatten norms of Hankel operators correspond to Besov norms of the symbols
(see Section~\ref{sec.b}), 
one can view the topic of this paper as the study of the map induced by \eqref{r} between Besov classes. 
We prefer to use an operator theoretic viewpoint whenever possible, although sometimes we 
have to resort to proofs in terms of Besov classes. 

\section{Preliminaries}\label{sec.b}

Throughout this paper, the symbol `$C$' with a (possibly empty) set of subscripts
will denote a positive constant, depending only on the subscripts, whose precise
value may change with each occurrence. Moreover, we write $X\asymp Y$ for two
expressions $X$ and $Y$ if $X\le C Y$ and $Y\le C X$.

\subsection{Operator theory, Schatten classes}
For a bounded linear  operator $A$ in a Hilbert space, we denote by $\norm{A}_{\calB}$ the operator norm of $A$. 

For a compact operator $A$ in a Hilbert space, let $\{s_n(A)\}_{n=1}^\infty$ 
be the sequence of singular values of $A$, enumerated with multiplicities 
taken into account. For $0<p<\infty$, the standard Schatten class $\Sch_p$ 
of compact operators is defined by the condition 
$$
A\in \Sch_p 
\quad\Leftrightarrow\quad
\norm{A}_{\Sch_p}^p:=\sum_{n\geq1} s_n(A)^p<\infty.
$$
$\norm{\cdot}_{\Sch_p}$ is a norm on $\Sch_p$ for $p\geq1$ and a
quasinorm for $0<p<1$.

\subsection{Characterisation of Schatten class Hankel operators}\label{sec.b2}

Let $\bbT$ deonte the unit circle. We consider the Fourier transform $\calF$ as
the unitary map from  $L^2(\bbT)$ to $\ell^2(\bbZ)$, 
$$
(\calF f)(j)=\wh f(j)=
\int_0^1 f(e^{2\pi it})e^{-2\pi ijt} \,dt, \quad j\in\bbZ.
$$
We also use its inverse $\calF^{-1}: \ell^2(\bbZ)\to L^2(\bbT)$ and denote $\wc \alpha=\calF^{-1}\alpha$. 
Similarly, we use the Fourier integral transform $\boldsymbol\calF$ in $L^2(\bbR)$ and its inverse
$$
(\boldsymbol\calF^{-1}\f)(\xi)=\wc\f(\xi)=\int_{\bbR} \f(t)e^{2\pi i\xi t}\,dt,
\quad \xi\in\bbR.
$$

Let $\bw\in C^\infty(\bbR)$ be a non-negative function such that $\supp \bw\subset [1/2,2]$ 
and 
$$
\sum_{m\in\bbZ} \bw(t/2^m)=1, \quad t>0.
$$
We set $\bw_m(t)=\bw(t/2^m)$. For $m\geq0$, we denote by $w_m$ the restriction of the function $\bw_m$ onto 
$\bbZ_+$, i.e. $w_m(j)=\bw_m(j)$, $j\geq0$. 
\begin{proposition}\cite[Theorem 6.7.4]{Peller}\label{thm.peller}
Let $0<p<\infty$. 
\begin{enumerate}[\rm (i)]
\item
For a bounded Hankel matrix $H(\alpha)$, one has
$$
\norm{H(\alpha)}_{\Sch_p}^p
\asymp
\abs{\alpha(0)}^p +\sum_{m\geq0} 2^m 
\norm{\calF^{-1}(\alpha w_m)}_{L^p(\bbT)}^p. 
$$
\item
For a bounded integral Hankel operator $\bH(\ba)$ one has
$$
\norm{\bH(\ba)}_{\Sch_p}^p 
\asymp
\sum_{m\in \bbZ}2^m \norm{\boldsymbol\calF^{-1}(\ba \bw_m)}_{L^p(\bbR)}^p. 
$$
\end{enumerate}
\end{proposition}
The expressions in the right side here are exactly the norms of the symbols in
the Besov class $B_p^{1/p}$.


\subsection{Periodization operator}\label{sec.b3}
Here we discuss the map induced by \eqref{r1}. 
For a compactly supported function $\f\in C(\bbR)$, we define the \emph{periodization} of $\f$ as the
function on the unit circle given by 
\[
\calP \f(e^{2\pi it}) = \sum_{j\in\bbZ} \f(t - j),
\quad e^{2\pi it} \in\bbT.
\label{b3.1}
\]
We call $\calP$ the \emph{periodization operator}. Applying the ``triangle
inequality" $\abs{a+b}^p\leq \abs{a}^p+\abs{b}^p$ for $0<p\leq1$ to
\eqref{b3.1} and then integrating over $t$ we see that
$$
\norm{\calP\f}_{L^p(\bbT)} \le \norm{\f}_{L^p(\bbR)},
\quad 0<p\le 1.
$$
This allows one to extend $\calP$ to a map from $L^1(\bbR)$ to $L^1(\bbT)$. 
For $\f\in L^1(\bbR)$  it is straightforward to see that
$$
\wh{\calP\f}(j) = \wh\f(j), \quad j\in \bbZ.
$$
Thus we have the estimate
\[
\Norm{\sum_{j\in\bbZ}\wh \f(j)z^j}_{L^p(\bbT)}
\le \norm{\f}_{L^p(\bbR)}, \quad 0<p\le 1.
\label{b3.3}
\]

\section{Pointwise restrictions}\label{sec.c}

\subsection{Pointwise restrictions for operators of class $\Sch_p$}\label{sec.c.1}

For $\lambda>0$, let $\delta_\lambda(t)=\delta(t-\lambda)$, where $\delta(t)$ is the
Dirac delta function, so that if $\ba\in C(0,\infty)$, then
$$
(\calR_{\delta_\lambda}\ba)(j) = \ba(j+\lambda), \quad j\ge 0. 
$$
If $\ba$ is the kernel function of an integral Hankel operator
of class $\Sch_1$, then $\ba$ is almost everywhere equal to a continuous
function on $(0,\infty)$ \cite[Corollary 7.10]{Partington}, and the estimate
\begin{equation}
\abs{\ba(t)}\leq C\norm{\bH(\ba)}_{\Sch_1}/t, \quad t>0,
\label{cc1}
\end{equation}
holds true with some absolute constant $C$. 
Thus, the
definition of $\calR_{\delta_\lambda}\ba$ makes sense without any further
restriction on $\ba$. 

The aim of this section is to prove the following.
\begin{theorem}\label{pointwise}
Let $0<p\leq1$, $\lambda>0$. If $\bH(\ba)\in\Sch_p$ then $H(\calR_{\delta_\lambda}\ba)\in\Sch_p$
and 
\[
\norm{H(\calR_{\delta_\lambda}\ba)}_{\Sch_p}
\leq 
C_p(1+1/\lambda)
\norm{\bH(\ba)}_{\Sch_p}.
\label{pointwise1}
\]
\end{theorem}

The main component in the proof of Theorem \ref{pointwise} is the estimate
\eqref{b3.3}.

\begin{proof}
Denote $\bb(t)=\ba(t+\lambda)$. 
By Proposition~\ref{thm.peller}(i), we have
\[
\norm{H(\calR_{\delta_\lambda}\ba)}_{\Sch_p}^p 
\leq 
C_p \abs{\bb(0)}^p+C_p\sum_{m\geq0} 2^m \Norm{\sum_{j\geq0}\bb(j)w_m(j)z^j}_{L^p(\bbT)}^p.
\label{c0}
\]
Let us first estimate the series in the right hand side of \eqref{c0}. 
Applying \eqref{b3.3} to $\f=\boldsymbol\calF^{-1}(\bb \bw_m)$, we obtain 
$$
\Norm{\sum_{j\geq0}\bb(j)w_m(j)z^j}_{L^p(\bbT)}^p
\leq 
\norm{\boldsymbol\calF^{-1}(\bb \bw_m)}_{L^p(\bbR)}^p
$$
for every $m\geq0$. 
By Proposition~\ref{thm.peller}(ii), this yields
$$
\sum_{m\geq0} 2^m \Norm{\sum_{j\geq0}\bb(j)w_m(j)z^j}_{L^p(\bbT)}^p
\leq
\sum_{m\in\bbZ} 2^m \norm{\boldsymbol\calF^{-1}(\bb \bw_m)}_{L^p(\bbR)}^p
\leq
C_p \norm{\bH(\bb)}_{\Sch_p}^p. 
$$
Let us relate the norm of $\bH(\bb)$ to the norm of $\bH(\ba)$. 
Writing 
$$
\int_0^\infty \int_0^\infty \bb(t+s)\f(t)\overline{\bg(s)}dt\, ds
=
\int_{\lambda/2}^\infty \int_{\lambda/2}^\infty 
\ba(t+s)\f(t-\lambda/2)\overline{\bg(s-\lambda/2)}dt\, ds, 
$$
we see that $\bH(\bb)$ is unitarily equivalent to the restriction of $\bH(\ba)$ onto 
the subspace $L^2(\lambda/2,\infty)\subset L^2(0,\infty)$.
It follows that 
\[
\norm{\bH(\bb)}_{\Sch_p}\leq \norm{\bH(\ba)}_{\Sch_p}
\label{c6}
\]
for all $p>0$. 
Finally, consider the first term in the right hand side of \eqref{c0}. 
By \eqref{cc1} we have
$$
\abs{\bb(0)}=
\abs{\ba(\lambda)}\leq C\norm{\bH(\ba)}_{\Sch_1}/\lambda\leq C\norm{\bH(\ba)}_{\Sch_p}/\lambda.
$$
Combining the above estimates, we arrive at the required statement. 
\end{proof}

\begin{remark*}
One can also consider restrictions of $\ba$ to the scaled lattice $\{\gamma j+\lambda\}_{j\geq0}$
for some $\gamma>0$. 
For $\gamma>0$, let $\ba_\gamma(t) = \ba(\gamma t)$ and let $V_\gamma:L^2(0,\infty)
\to L^2(0,\infty)$ be the unitary operator
$$
V_\gamma \f(t) = \sqrt{\gamma} \f(\gamma t), \quad t>0.
$$
Then $\gamma\bH(\ba_\gamma) = V_\gamma \bH(\ba) V_\gamma^*$ and so
$\gamma\norm{\bH(\ba_\gamma)}_{\Sch_p} = \norm{\bH(\ba)}_{\Sch_p}$ for all $0<p<\infty$.
It follows from this and Theorem \ref{pointwise} that if $0<p\le 1$, then
$\gamma\norm{H(\calR_{\delta_\lambda}\ba_\gamma)}_{\Sch_p}
\le C_{p} \norm{\bH(\ba)}_{\Sch_p}$, and thus
\[
\sup_{\gamma>0}\gamma\norm{H(\calR_{\delta_\lambda}\ba_\gamma)}_{\Sch_p}
\le C_{p} \norm{\bH(\ba)}_{\Sch_p}.
\label{c11}
\]
\end{remark*}

\subsection{Counterexample for $p>1$}

For $p>1$, is it no longer the case that the kernel of an integral
Hankel operator of class $\Sch_p$ is necessarily continuous. However,
even if we restrict to operators with continuous kernels, the conclusions
of Theorem~\ref{pointwise} still fail and thus the condition $0<p\le 1$
is sharp.

To show this, we fix a smooth kernel function $\ba$ with $\supp \ba\subset [1/2,2]$ and $\ba(1)=1$
and let $\ba^{(N)}(t)=\ba(1+N(t-1))$ for $N\in\bbN$. 
Then for each $N$ we have
\begin{align*}
\calR_{\delta_1}\ba^{(N)}(0)&=\ba^{(N)}(1)=1,
\\
\calR_{\delta_1}\ba^{(N)}(j)&=\ba^{(N)}(1+j)=0,\quad j\geq1.
\end{align*}
It follows that
$$
\norm{H(\calR_{\delta_1}\ba^{(N)})}_{\Sch_p}=1
$$
for all $p\geq 1$ and $N\in\bbN$. 
On the other hand, it is not difficult to show that 
$$
\norm{\bH(\ba^{(N)})}_{\Sch_p}^p\leq CN^{1-p}
$$
which tends to zero as $N\to\infty$ whenever $p>1$. 
Indeed, by the assumption on the support of $\ba$ we have
$$
\ba^{(N)}=\ba^{(N)} \bw_{-1}+\ba^{(N)} \bw_{0}+\ba^{(N)} \bw_{1}
$$
for all $N$, where $\bw_m$ are defined in Section \ref{sec.b2}. 
It is easy to conclude that
$$
\sum_{m=-1,0,1}2^m\norm{\boldsymbol\calF^{-1}(\ba^{(N)}\bw_m)}_{L^p(\bbR)}^p
\leq
C_p\norm{\boldsymbol\calF^{-1}(\ba^{(N)})}_{L^p(\bbR)}^p=CN^{1-p}.
$$
\subsection{Partial converse of Theorem~\ref{pointwise}}\label{sec.c.2}
It is clear that one cannot bound  $\bH(\ba)$ by $H(\calR_{\delta_\lambda}\ba)$ in any norm. 
However, one can achieve a partial converse if we vary our restriction operators
in an appropriate sense and take a supremum over all restrictions in the right side. 
Here we briefly sketch a sample argument of this nature. 
Fix $0<p\leq1$; 
we use ``continuous'' counterparts of the expressions in Proposition~\ref{thm.peller}; 
see e.g. \cite[Section 2.3.3, p.99]{Triebel2}:
\begin{align}
\norm{H(\alpha)}_{\Sch_p}^p
&\asymp
\abs{\alpha(0)}^p +\int_0^2 
\norm{\calF^{-1}(\alpha w^\tau)}_{L^p(\bbT)}^p\,
\frac{d\tau}{\tau^2}, 
\label{d.1}
\\
\norm{\bH(\ba)}_{\Sch_p}^p 
&\asymp
\int_0^\infty \norm{\boldsymbol\calF^{-1}(\ba \bw^\tau)}_{L^p(\bbR)}^p
\frac{d\tau}{\tau^2}, 
\notag
\end{align}
where $\bw^\tau(t)=\bw(\tau t)$ and $w^\tau=\{\bw(\tau j)\}_{j\ge 0}$.

Let $\ba$ be a continuous function on $(0,\infty)$ and, for $\gamma>0$,
let $\ba_\gamma(t) = \ba(\gamma t)$. Observe that, by a change of variable, 
\begin{align*}
\norm{\calF^{-1}(w^\tau\calR_{\delta_\lambda}\ba_\gamma)}_{L^p(\bbT)}^p
&=
\int_{-1/2}^{1/2}\left|
\sum_{j\ge 0} \bw(\tau j)\ba(\gamma (j+\lambda))
e^{2\pi ij s} \right|^p\,ds
\\
&=
\gamma^{1-p}\int_{-1/2\gamma}^{1/2\gamma}\left|
\gamma\sum_{j\ge 0} \bw(\tau j)\ba(\gamma (j+\lambda))
e^{2\pi ij\gamma s} \right|^p\,ds.
\end{align*}
By another change of variable, 
it then follows from \eqref{d.1} that
\begin{align}
\gamma^p\norm{H(\calR_{\delta_\lambda}\ba_\gamma)}_{\Sch_p}^p
&\ge C_p
\gamma\int_0^2\int_{-1/2\gamma}^{1/2\gamma}\left|
\gamma\sum_{j\ge 0} \bw(\tau j)\ba(\gamma (j+\lambda))
e^{2\pi ij\gamma s} \right|^p\,ds\frac{d\tau}{\tau^2} \nonumber\\
&= C_p
\int_0^{2/\gamma}\int_{-1/2\gamma}^{1/2\gamma}\left|
\gamma\sum_{j\ge 0} \bw(\tau \gamma j)\ba(\gamma (j+\lambda))
e^{2\pi ij\gamma s} \right|^p\,ds\frac{d\tau}{\tau^2}.
\label{Riemann sum}
\end{align}
Since $\ba$ is continuous, for each $s\in\bbR$ and $\tau>0$ the integrand
in \eqref{Riemann sum} converges to $|\boldsymbol\calF^{-1}(\ba \bw^\tau)(s)|^p$
as $\gamma\to 0$. Then by Fatou's Lemma we see that
\begin{align}
\norm{\bH(\ba)}_{\Sch_p}^p
&\asymp
\int_0^\infty \norm{\boldsymbol\calF^{-1}(\ba \bw^\tau)}_{L^p(\bbR)}^p
\frac{d\tau}{\tau^2} \nonumber\\
&\le \lim_{\gamma\to 0}
\int_0^{2/\gamma}\int_{-1/2\gamma}^{1/2\gamma}\left|
\gamma\sum_{j\ge 0} \bw(\tau \gamma j)\ba(\gamma (j+\lambda))
e^{2\pi ij\gamma s} \right|^p\,ds\frac{d\tau}{\tau^2} \nonumber\\
&\le C_p \lim_{\gamma\to 0}
\gamma^p\norm{H(\calR_{\delta_\lambda}\ba_\gamma)}_{\Sch_p}^p.
\label{d.3}
\end{align}
This gives an analogue of Igari's theorem for Fourier multipliers \cite{Igari}.
Combining \eqref{d.3} with \eqref{c11} gives the estimate
$$
\norm{\bH(\ba)}_{\Sch_p}
\asymp
\sup_{\gamma>0}\gamma\norm{H(\calR_{\delta_\lambda}\ba_\gamma)}_{\Sch_p},
\quad 
0<p\leq1.
$$

\section{Pointwise restriction for non-negative operators}

\subsection{Statement of the result}

Although Theorem~\ref{pointwise} fails for $p>1$, the estimate \eqref{pointwise1}
remains valid for all $0<p<\infty$ if we restrict to the class of non-negative
operators (in the usual quadratic form sense). 
Before stating this precisely we recall (see e.g. \cite[page 22]{Widom}) that a 
bounded integral Hankel operator $\bH(\ba)$ is non-negative 
 if and only if the kernel function $\ba$ can be represented as
\[
\ba(t)=\int_0^\infty e^{-t\eta}d\mu(\eta),
\label{c4}
\]
where the measure $\mu$ satisfies
$$
\mu((0,\eta)) \leq C\eta, \quad \eta>0.
$$
In particular, it follows that the kernel function $\ba(t)$ is continuous in $t>0$, 
and therefore the restriction $\calR_{\delta_\lambda}\ba$ is well defined for all $\lambda>0$.

We have the following theorem.

\begin{theorem}\label{positive}
Let $\bH(\ba)\ge 0$ be a bounded integral Hankel operator and $\lambda>0$. 
Then the following hold:
\begin{enumerate}[\rm (i)]
\item 
If $\bH(\ba)\in\calB$ then $H(\calR_{\delta_\lambda}\ba)\in\calB$ and
\[
\norm{H(\calR_{\delta_\lambda}\ba)}_\calB
\leq
C(1+1/\lambda)
\norm{\bH(\ba)}_{\calB}. 
\label{c7}
\]

\item
If $\bH(\ba)\in\Sch_p$ for some $0<p\leq\infty$, then
$H(\calR_{\delta_\lambda}\ba)\in\Sch_p$ and
\[
\norm{H(\calR_{\delta_\lambda}\ba)}_{\Sch_p} 
\leq 
C_{p}(1+1/\lambda) \norm{\bH(\ba)}_{\Sch_p}.
\label{c8}
\]
\end{enumerate}
\end{theorem}

\begin{remark}
\begin{enumerate}
\item
Observe that by \eqref{c4}, the kernel function 
$\ba$ is necessarily positive, monotone decreasing and continuous
on $(0,\infty)$. In fact, the proof of Theorem~\ref{positive} depends only on these properties of $\ba$. 
\item
If $\lambda\geq2$, one can slightly improve the statement of Theorem~\ref{positive}. 
In this case one gets
\begin{align*}
\norm{H(\calR_{\delta_\lambda}\ba)}_\calB &\leq \norm{\bH(\ba)}_\calB,
\quad \lambda\geq2,
\\
\norm{H(\calR_{\delta_\lambda}\ba)}_{\Sch_p} &\leq \norm{\bH(\ba)}_{\Sch_p}, 
\quad \lambda\geq2, \quad p\in 2\bbN,
\end{align*}
\end{enumerate}
i.e. the constants in the estimates are equal to one in these cases. 
\end{remark}

In the rest of this section we prove Theorem~\ref{positive}. 
Observe that we only need to consider the case $p>1$, as for $0<p\leq1$ the required result follows from Theorem~\ref{pointwise}. 

Our proof consists of two different parts.
The first one is a short operator theoretic argument based on pointwise domination which however works only for $p\in2\bbN$ or $p=\infty$. 
The second one is a direct calculation based on Proposition~\ref{thm.peller} which applies to all $p\geq1$.

\subsection{Proof for  $p\in2\bbN\cup\{\infty\}$}

First we need a version of \eqref{cc1} for non-negative operators. 
\begin{lemma}\label{lma.c1}
Let $\bH(\ba)\ge 0$ be a bounded integral Hankel operator and $\lambda>0$. 
Then 
$$
\lambda\ba(\lambda)\leq 2\norm{\bH(\ba)}_\calB. 
$$
\end{lemma}
\begin{proof}
Take $\f(t)=e^{-t/\lambda}$; then $\norm{\f}^2_{L^2(0,\infty}=\lambda/2$ 
and using the monotonicity of $\ba(t)$, 
\begin{multline*}
(\bH(\ba)\f,\f)
=
\int_0^\infty \int_0^\infty \ba(t+s)e^{-(t+s)/\lambda}dt\, ds
=
\int_0^\infty \ba(t)e^{-t/\lambda}tdt
\\
\geq
\int_0^\lambda \ba(t)e^{-t/\lambda}tdt
\geq
\ba(\lambda)\int_0^\lambda e^{-t/\lambda}tdt
=(1-2e^{-1})\lambda^2\ba(\lambda). 
\end{multline*}
On the other hand, 
$$
(\bH(\ba)\f,\f)\leq \norm{\bH(\ba)}_\calB\norm{\f}^2_{L^2(0,\infty)}
=(\lambda/2) \norm{\bH(\ba)}_\calB. 
$$
Combining these two estimates, we obtain 
$$
\lambda\ba(\lambda)\leq
\frac{1}{2(1-2e^{-1})}\norm{\bH(\ba)}_\calB\leq 2\norm{\bH(\ba)}_\calB,
$$
as required. 
\end{proof}

\begin{proof}[Proof of Theorem~\ref{positive} for $p\in2\bbN\cup\{\infty\}$]

First let us \textbf{assume that $\lambda\geq2$.}
Let $K$ be the integral operator in $L^2(0,\infty)$ with the integral kernel
$$
K(t,s)=\ba(\lambda+\floor{t}+\floor{s}),
$$
where $\floor{t}$ is the largest integer less than or equal to $t$. 
Since 
$$
\lambda+\floor{t}+\floor{s}\geq \lambda+(t-1)+(s-1)\geq t+s,
$$ 
by monotonicity of $\ba$ we have
$$
K(t,s)\leq \ba(t+s).
$$
In the terminology of \cite[Chapter 2]{Simon}, this means that $K$ is 
\emph{pointwise dominated} by $\bH(\ba)$.  
By \cite[Theorem 2.13]{Simon}, it follows that 
$$
\norm{K}\leq \norm{\bH(\ba)}_\calB \quad \text{ and} \quad 
\norm{K}_{\Sch_p}\leq \norm{\bH(\ba)}_{\Sch_p}
$$
for all $p\in 2\bbN$. 
(This implication does not extend to $p\not\in 2\bbN$; see e.g. \cite{Peller2,Simon2}.) 
It is also true (see \cite{Pitt,DoddsFremlin}) that the compactness of $\bH(\ba)$ 
implies the compactness of $K$. 

Next, let us relate $K$ to $H(\calR_{\delta_\lambda}\ba)$. 
For $\f\in L^2(0,\infty)$ let us write the quadratic form of $K$ as
$$
(K\f,\f)=\sum_{j,k\geq0}\ba(\lambda+j+k)f_j \overline{f_k}, 
\qquad
f_j=\int_j^{j+1}\f(t)dt.
$$
This means that, writing 
$L^2(0,\infty)=\ell^2(\bbZ_+)\otimes L^2(0,1)$, the operator $K$ can be represented
as 
$$
K=H(\calR_{\delta_\lambda}\ba)\otimes (\cdot,\1)\1,
$$
where $(\cdot,\1)\1$ is the rank one operator in $L^2(0,1)$ acting as 
$$
\f\mapsto \int_0^1\f(t)dt. 
$$
It follows that 
$$
\norm{K}_\calB=\norm{H(\calR_{\delta_\lambda}\ba)}_\calB
\quad\text{ and }\quad
\norm{K}_{\Sch_p}=\norm{H(\calR_{\delta_\lambda}\ba)}_{\Sch_p}
$$
for all $p>0$. This completes the proof for $\lambda\geq2$ and $p\in 2\bbN\cup\{\infty\}$.

Let us consider the case $0<\lambda<2$. 
Let $P_2$ be the projection onto $\ell^2(\{2,3,\dots\})$ in $\ell^2(\bbZ_+)$. 
Write
$$
H(\calR_{\delta_\lambda}\ba)=P_2H(\calR_{\delta_\lambda}\ba)P_2+\wt H.
$$
The operator $\wt H$ is of rank $\leq4$. Inspecting the matrix elements of $\wt H$ and  using Lemma~\ref{lma.c1}, it is easy to see
that 
$$
\norm{\wt H}_{\Sch_1}\leq C\norm{\bH(\ba)}_{\calB}/\lambda, \quad \lambda>0.
$$
On the other hand, 
the operator $P_2 H(\calR_{\delta_\lambda}\ba)P_2$ is unitarily equivalent to 
$H(\calR_{\delta_{\lambda+2}}\ba)$. Thus, applying the previous step of the proof, we
obtain 
$$
\norm{P_2 H(\calR_{\delta_\lambda}\ba)P_2}_\calB
\leq
\norm{\bH(\ba)}_\calB
\quad \text{ and }\quad
\norm{P_2 H(\calR_{\delta_\lambda}\ba)P_2}_{\Sch_p}
\leq
\norm{\bH(\ba)}_{\Sch_p}
$$
for $p\in2\bbN$. Combining these estimates, we arrive at  
\eqref{c7} and \eqref{c8} for $p\in 2\bbN$. 
\end{proof}

As already mentioned, this proof does not  extend to $p\not\in 2\bbN$; see e.g. \cite{Peller2,Simon2}. 
Below we give a different proof which works for all $1\leq p<\infty$, but does not give
precise information about the constants in the estimates.

\subsection{Proof of Theorem~\ref{positive} for $1\leq p<\infty$}

In order to simplify our notation, we set $\bb(t)=\ba(t+\lambda)$, $b(k)=\ba(k+\lambda)$, and 
\begin{align*}
\wc b_m(z)&=\sum_{k\geq0}b(k)w_m(k)z^k, \quad m\in\bbZ_+, \quad z\in \bbT,
\\
\wc \bb_m(\xi)&=\int_0^\infty \bb(t) \bw_m(t) e^{2\pi i \xi t} dt, \quad m\in \bbZ, \quad \xi\in\bbR.
\end{align*}
The core of the proof is the bound
\[
\sum_{m\geq1} 2^m \norm{\wc b_m}_{L^p(\bbT)}^p
\leq
C_p
\sum_{m\in\bbZ}2^m \norm{\wc \bb_m}_{L^p(\bbR)}^p,
\label{c5}
\]
which we prove below. 
Throughout the proof, we use the property that $\bb$ and $b$ are positive and monotone decreasing.

\textbf{First step: upper bound for $\norm{\wc b_m}_{L^p(\bbT)}$.}
Fix $m\geq1$. 
First we prepare two pointwise bounds for $\wc b_m(z)$. The first one is trivial:
\[
\abs{\wc b_m(z)}
\leq 
\sum_k b(k)w_m(k)
\leq 
2^{m+1}b(2^{m-1}).
\label{c9}
\]
The second one is obtained through a discrete version of integration by parts (Abel summation). 
We have
\begin{multline*}
\wc b_m(z)
=
\frac{1}{z-1}\sum_k b(k)w_m(k)(z^{k+1}-z^k)
\\
=
\frac{1}{z-1}\sum_k \bigl(b(k)w_m(k)-b(k+1)w_m(k+1)\bigr)z^{k+1}
\\
=
\frac{1}{z-1}\sum_k \bigl((b(k)-b(k+1))w_m(k)+b(k+1)(w_m(k)-w_m(k+1))\bigr)z^{k+1},
\end{multline*}
and therefore 
\begin{multline*}
\abs{\wc b_m(z)}
\leq
\frac{1}{\abs{z-1}}\sum_{k=2^{m-1}}^{2^{m+1}} (b(k)-b(k+1)) \\
+
\frac{1}{\abs{z-1}}\sum_{k=2^{m-1}}^{1+2^{m+1}} b(k)\abs{w_m(k-1)-w_m(k)}.
\end{multline*}
Clearly, the first sum here is telescoping. For the second sum, we use the estimate
$$
\abs{w_m(k-1)-w_m(k)}\leq C2^{-m}. 
$$
Putting this together, we obtain
\begin{multline}
\abs{\wc b_m(z)}
\leq
\frac{1}{\abs{z-1}}
(b(2^{m-1})-b(1+2^{m+1})) \\
+
\frac{C}{\abs{z-1}}2^{-m}\sum_{k=2^{m-1}}^{1+2^{m+1}} b(k)
\leq
\frac{C}{\abs{z-1}}b(2^{m-1}),
\label{c10}
\end{multline}
which is our  second bound for $\wc b_m(z)$. 

Now we can estimate the norm $\norm{\wc b_m}_{L^p(\bbT)}$. 
We split the integral over the unit circle into two parts and estimate them separately. 
Using \eqref{c9}, we obtain
$$
2^m\int_{\abs{t}<2^{-m}}\abs{\wc b_m(e^{2\pi it})}^pdt
\leq 
C2^{pm}b(2^{m-1})^p.
$$
Using \eqref{c10}, we get
\begin{multline*}
2^m\int_{\abs{t}>2^{-m}}\abs{\wc b_m(e^{2\pi it})}^p dt
\leq 
C2^m\int_{\abs{t}>2^{-m}}\frac{dt}{\abs{e^{2\pi it}-1}^p}b(2^{m-1})^p
\\
\leq
C2^m\int_{2^{-m}}^1\frac{dt}{t^p}b(2^{m-1})^p
\leq
C2^{pm}b(2^{m-1})^p. 
\end{multline*}
Combining the estimates for two integrals above, we obtain
$$
2^m\norm{\wc b_m}^p_{L^p(\bbT)}\leq C2^{pm}b(2^{m-1})^p.
$$

\textbf{Second step: lower bound for $\norm{\wc \bb_m}_{L^p(\bbR)}$.}
For the derivative of $\bb_m$ we have
$$
\wc \bb_m'(\xi)=2\pi i\int_0^\infty \bb(t)\bw_m(t)te^{2\pi it\xi}dt,
$$
and therefore
$$
\abs{\wc \bb_m'(\xi)}
\leq
2\pi 
\int_0^\infty \bb(t)\bw_m(t)tdt
\leq 
2^{m+2}\pi \int_0^\infty \bb(t)\bw_m(t)dt
=
2^{m+2}\pi \wc \bb_m(0).
$$
It follows that 
$$
\abs{\wc \bb_m(\xi)-\wc \bb_m(0)}\leq \abs{\xi}2^{m+2}\pi \wc\bb_m(0),
$$
and therefore for $\abs{\xi}<2^{-m-5}$ we have 
$$
\abs{\wc \bb_m(\xi)}\geq \wc\bb_m(0)/2.
$$
We use this to obtain a lower bound for the integral of $\abs{\wc \bb_m}^p$: 
\begin{multline*}
2^m\int_\bbR\abs{\wc\bb_m(\xi)}^pd\xi
\geq 
2^m \int_{\abs{\xi}<2^{-m-5}}\abs{\wc\bb_m(\xi)}^pd\xi
\geq
2^{-5}( \wc\bb_m(0)/2)^p=C\wc\bb_m(0)^p.
\end{multline*}
Finally,
$$
\wc\bb_m(0)
=
\int_0^\infty \bb(t)\bw_m(t)dt
\geq
\bb(2^{m+1})\int_0^\infty \bw_m(t)dt
=C2^{m}\bb(2^{m+1}),
$$
and so we obtain 
$$
2^m\norm{\wc \bb_m}_{L^p(\bbR)}^p\geq C 2^{mp}b(2^{m+1})^p.
$$

\textbf{Combining the two steps and completing the proof.}
Combining the upper bound for $\norm{\wc b_m}_{L^p(\bbT)}$ and the lower bound
for $\norm{\wc\bb_m}_{L^p(\bbR)}$, we obtain
$$
2^m\norm{\wc b_m}_{L^p(\bbT)}^p
\leq 
C2^{p(m-1)}b(2^{m-1})^p
\leq
C2^{m-2}\norm{\wc \bb_{m-2}}_{L^p(\bbR)}^p, \quad m\geq1.
$$
Summing over $m$, we obtain the bound \eqref{c5}. 

By Proposition~\ref{thm.peller}(i), we have
\[
\norm{H(b)}_{\Sch_p}^p
\leq
C_p\abs{b(0)}^p+
C_p\sum_{m\geq0}  2^m \norm{\wc b_m}_{L^p(\bbT)}^p.
\label{dd}
\]
By Lemma~\ref{lma.c1}, we have
$$
\abs{b(0)}^p=\abs{\ba(\lambda)}^p\leq 2^p \norm{\bH(\ba)}_\calB^p/\lambda^p. 
$$
Similarly, the $m=0$ term in the series in \eqref{dd} can be estimated as follows:
$$
\norm{\wc b_0}_{L^p(\bbT)}^p=\abs{b(1)}^p=\abs{\ba(\lambda+1)}^p
\leq 2^p\norm{\bH(\ba)}_\calB^p/(1+\lambda)^p
\leq 2^p\norm{\bH(\ba)}_\calB^p/\lambda^p.
$$
Combining this with \eqref{c5} and using 
 Proposition~\ref{thm.peller}(ii), we obtain 
$$
\norm{H(b)}_{\Sch_p}^p
\leq
C_p
\norm{\bH(\ba)}_\calB^p/\lambda^p
+
C_p\sum_{m\in\bbZ}2^m \norm{\wc \bb_m}_{L^p(\bbR)}^p
\leq
C_p
\norm{\bH(\ba)}_\calB^p/\lambda^p
+
C_p\norm{\bH(\bb)}_{\Sch_p}^p.
$$
Finally, as in \eqref{c6}, we have $\norm{\bH(\bb)}_{\Sch_p}\leq \norm{\bH(\ba)}_{\Sch_p}$, 
and  we arrive at the required estimate \eqref{c8}.

\section{Restriction by averaging}\label{sec.d}

\subsection{Boundedness of restrictions by averaging}

The main result of this section says that if the function $\bvphi$ is sufficiently regular,
then the map $\bH(\ba)\mapsto H(\calR_\bvphi\ba)$ is bounded with respect to all
Schatten norms. We will make use of the periodisation operator $\calP$ from Section
\ref{sec.b3}.

\begin{theorem}\label{average}
	Let $\bvphi\in C(\bbR)$ be such that $\supp\bvphi\subset[0,\infty)$ and 
	$\calP(\abs{\wc\bvphi})\in L^\infty(\bbT)$. 
	Then there exist bounded operators $\Phi_1$ and $\Phi_2$ acting from
	$\ell^2(\bbZ_+)$ to $L^2(0,\infty)$ such that
	\[
	\Phi_2^*\bH(\ba)\Phi_1=H(\calR_\bvphi\ba)
	\label{average2}
	\]
	and $\norm{\Phi_1}_\calB=\norm{\Phi_2}_\calB =\sqrt{A}$, where $A=\norm{\calP(\abs{\wc\bvphi})}_{L^\infty(\bbT)}$. 
	Consequently, we have
$$
\norm{H(\calR_{\bvphi}\ba)}_{\calB}\leq A\norm{\bH(\ba)}_{\calB}
\quad \text{ and }\quad
\norm{H(\calR_{\bvphi}\ba)}_{\Sch_p}\leq A\norm{\bH(\ba)}_{\Sch_p}
$$
for every $0<p<\infty$. 
\end{theorem}

A close inspection of the proof of Theorem \ref{average} will reveal that the
condition $\calP(\abs{\wc\bvphi})\in L^\infty(\bbT)$ is necessary, in the sense that if there
exist bounded operators $\Phi_1,\Phi_2:\ell^2(\bbZ_+)\to L^2(0,\infty)$ such
that \eqref{average2} holds, then $\calP(\abs{\wc\bvphi})\in L^\infty(\bbT)$.

It will be convenient to separate the statement related to the boundedness of the maps
$\Phi_1$ and $\Phi_2$. 
\begin{lemma}\label{lma.d2}
	Let $\bpsi\in L^2(\bbR)$ with $\calP(\abs{\wc\bpsi}^2) \in L^\infty(\bbT)$.
	Then the map 
	$$
	\Phi: x=\{x(j)\}_{j\geq0}\mapsto \sum_{j\geq0} x(j) \bpsi(t-j), \quad t\in\bbR,
	$$
	is bounded from $\ell^2(\bbZ_+)$ to $L^2(\bbR)$, with
	$\norm{\Phi}_\calB^2=\norm{\calP(\abs{\wc\bpsi}^2)}_{L^\infty(\bbT)}$. 
\end{lemma}
\begin{proof}
	Let $x$ be a finitely supported sequence. 
	We have, using Parseval's theorem,
	\begin{multline*}
	\Norm{\Phi x}^2_{L^2(\bbR)} =
	\Norm{\sum_{j\geq0} x(j) \bpsi(\cdot - j)}^2_{L^2(\bbR)}
	= 
	\int_\bbR \Abs{\sum_{j\geq0}x(j) \wc\bpsi(\xi)e^{2\pi ij\xi}}^2d\xi
	\\
	=
	\int_\bbR \abs{\wc x(e^{2\pi i\xi})}^2\abs{\wc\bpsi(\xi)}^2d\xi
	=
	\sum_{j\in\bbZ}\int_0^1\abs{\wc x(e^{2\pi i\xi})}^2\abs{\wc\bpsi(\xi-j)}^2 d\xi
	\\
	=
	\int_0^1\abs{\wc x(e^{2\pi i\xi})}^2 \calP(\abs{\wc\bpsi}^2)(\xi) d\xi
	\leq
	\norm{\calP(\abs{\wc\bpsi}^2)}_{L^\infty(\bbT)}\norm{\wc x}_{L^2(\bbT)}^2
	=
	\norm{\calP(\abs{\wc\bpsi}^2)}_{L^\infty(\bbT)}\norm{x}_{\ell^2(\bbZ_+)}^2.
	\end{multline*}
	It is also clear that the inequality here is sharp in the sense that
	$$
	\sup_{\|\wc x\|_{L^2(\bbT)}=1}
	\int_\bbR \abs{\wc x(e^{2\pi i\xi})}^2\abs{\wc\bpsi(\xi)}^2d\xi
	=
	\norm{\calP(\abs{\wc\bpsi}^2)}_{L^\infty(\bbT)}. 
	$$
	This proves the claim. 
\end{proof}

Below $\bbC_+$ will denote the upper half-plane; 
$H^\infty(\bbC_+)$, $H^2(\bbC_+)$ etc. are the standard Hardy classes.
We will sometimes identify functions in these Hardy classes with their boundary
values on $\bbR$.

\begin{proof}[Proof of Theorem~\ref{average}]
By assumption, we have $\calP(\abs{\wc\bvphi})\in L^\infty(\bbT)\subset L^1(\bbT)$; 
it follows that $\wc\bvphi\in L^1(\bbR)$. 
Recalling that $\supp \bvphi\subset[0,\infty)$, 
we obtain that $\wc\bvphi\in H^1(\bbC_+)$. 
Thus, we can factorise $\wc\bvphi$ into a product of two $H^2(\bbC_+)$-functions. 
More precisely, there exist
	$\bvphi_1,\bvphi_2\in L^2(0,\infty)$ such that 
	$$
	\wc\bvphi(\xi)=\wc\bvphi_1(\xi)\overline{\wc\bvphi_2(-\xi)}
	\quad \text{ and }\quad
	\abs{\wc\bvphi_1(\xi)}=\abs{\wc\bvphi_2(-\xi)}, 
	\quad  \forall \xi\in\bbR.
	$$ 
	Then $\bvphi = \bvphi_1\ast\overline{\bvphi_2}$ and
	$$
	\norm{\calP(\abs{\wc\bvphi_1}^2)}_{L^\infty(\bbT)}=
	\norm{\calP(\abs{\wc\bvphi_2}^2)}_{L^\infty(\bbT)}=
	\norm{\calP(\abs{\wc\bvphi})}_{L^\infty(\bbT)}.
	$$
	Next, for $i=1,2$, let us define the map $\Phi_i:\ell^2(\bbZ_+)\to L^2(0,\infty)$ by
	$$
	\Phi_i: x=\{x(j)\}_{j\ge 0}\mapsto \sum_{j\geq0} x(j)
	\bvphi_i(\cdot - j). 
	$$
	By Lemma~\ref{lma.d2}, both $\Phi_1$ and $\Phi_2$ are bounded
	with norms equal to $\sqrt{\norm{\calP(\abs{\wc\bvphi})}_{L^\infty(\bbT)}}$. 
	
	In order to prove \eqref{average2}, let us first rearrange the definition
	of $\calR_\bvphi\ba$. For each $j,k\ge 0$
	
	\begin{multline*}
	\calR_\bvphi\ba(j+k)=\int_0^\infty \ba(t)\bvphi(t-j-k)\,dt
	=
	\int_0^\infty \ba(t) \int_\bbR \bvphi_1(t-s-j-k)\overline{\bvphi_2(s)}\,dt\, ds
	\\
	=
	\int\int_{s+t>0} \ba(s+t)\bvphi_1(t-j)\overline{\bvphi_2(s-k)}\,dt\, ds.
	\end{multline*}
	Since both $\bvphi_1$ and $\bvphi_2$ are supported on $(0,\infty)$, we
	can rewrite this as
	$$
	\calR_\bvphi\ba(j+k)=
	\int_0^\infty \int_0^\infty \ba(s+t)\bvphi_1(t-j)\overline{\bvphi_2(s-k)}\,ds\, dt.
	$$

	Now for $x=\{x(j)\}_{j\ge 0}\in \ell^2(\bbZ_+)$, let us compute
	the quadratic form
	\begin{multline*}
	(\bH(\ba)\Phi_1 x,\Phi_2 x)
	=
	\sum_{j,k\ge 0}\int_0^\infty \int_0^\infty \ba(t+s) x(j) \overline{x(k)}
	\bvphi_1(t-j)\overline{\bvphi_2(s-k)}\,dt\,ds
	\\
	=
	\sum_{j,k\ge 0}\calR_\bvphi\ba(j+k)x(j) \overline{x(k)},
	\end{multline*}
	which yields \eqref{average2}.
\end{proof}

\subsection{Unitary equivalence and restrictions associated to general convolutions}

Let $L_n = L_n^{(0)}$ be the $n$-th Laguerre polynomial (see \cite[Ch. V]{Szego}
for the definition) and let
\[
\bu_n (t) = -2i\sqrt{\pi}L_n(4\pi t)e^{-2\pi t}, \quad t>0.
\label{laguerre}
\]
Then $\{\bu_n\}_{n\ge 0}$ is an orthonormal basis of
$L^2(0,\infty)$. It is well known that the matrix of an integral Hankel operator
is a Hankel matrix in the basis $\{\bu_n\}_{n\ge 0}$ and hence the classes
of Hankel matrices and integral Hankel operators are unitarily equivalent
\cite[Ch. 1, Thm 8.9]{Peller}. 

In this subsection we discuss how this unitary equivalence fits into our  ``restriction by averaging" framework. 
This requires looking at restrictions by averaging of a more
general type than considered above. 
To a given integral Hankel operator $\bH(\ba)$ we
associate the Hankel matrix $H(\alpha)$ with
$$
\alpha_j=\int_0^\infty \ba(t)\bvphi_j(t)dt,\quad j\geq0,
$$
where $\bvphi_j$ is a certain sequence of smooth functions, 
a more general one than just translations of a single function. 
Our sequence $\bvphi_j$ will be given by the multiple convolution of the form
$$
\bvphi_j=\bvphi*\underbrace{\nu*\nu*\cdots*\nu}_{\text{$j$ terms}},\quad j\geq0,
$$
where $\bvphi$ is a sufficiently regular function supported on $[0,\infty)$, 
and $\nu$ is a positive finite measure supported on $[0,\infty)$. 
Observe that if $d\nu(t)=\delta(t-1)dt$, then $\bvphi_j(t)=\bvphi(t-j)$, 
so we recover the definition of $\calR_\varphi$.

To make the multiple convolution notation more readable, we introduce the 
(formal) convolution with $\nu$ operator
$$
T_\nu\f=\f*\nu;
$$
then $\bvphi_j=T_\nu^j \bvphi$.

\begin{theorem}\label{convolution}
Let $\nu$ be a positive measure on $[0,\infty)$ with $\nu([0,\infty))\leq1$, 
and let $\bvphi\in C(\bbR)$ satisfy $\supp\bvphi\subset[0,\infty)$ and 
\[
\abs{\wc \bvphi(\xi)} \le \frac{C}{1+\xi^2}, \quad \xi\in\bbR.
\label{convolution1}
\]
For $j\geq0$, set $\bvphi_j=T_\nu^j \bvphi$ and consider the map
$$
\ba(t)\mapsto \alpha=\{\alpha(j)\}_{j=0}^\infty, 
\quad
\alpha(j)=\int_0^\infty \ba(t)\bvphi_j(t) dt. 
$$
Then there exist bounded operators $\Phi_1$ and $\Phi_2$ acting
from $\ell^2(\bbZ_+)$ to $L^2(0,\infty)$ such that
\[
\Phi_2^*\bH(\ba)\Phi_1=H(\alpha). 
\label{d2}
\]
Consequently, 
$$
\norm{H(\alpha)}\leq A\norm{\bH(\ba)}
\quad\text{ and }\quad
\norm{H(\alpha)}_{\Sch_p}\leq A\norm{\bH(\ba)}_{\Sch_p}
$$
for all $0<p<\infty$, where $A=\norm{\Phi_1}\norm{\Phi_2}$. 
\end{theorem}

It will again be convenient to separate the boundedness of $\Phi_1$ and $\Phi_2$
into a lemma.

\begin{lemma}\label{lma.d3}
Let $\omega\in H^\infty(\bbC_+)$ with $\norm{\omega}_{H^\infty}\leq 1$. 
Then the map 
\[
x=\{x(j)\}_{j=0}^\infty\mapsto \sum_{j\geq 0} \frac{x(j)\omega(\xi)^j}{\xi+i}, \quad \xi\in \bbC_+,
\label{d1}
\]
is bounded from $\ell^2(\bbZ_+)$ to $H^2(\bbC_+)$. 
\end{lemma}
\begin{proof}
Consider the conformal map 
$$
\bbD\ni\zeta\mapsto \xi=i\frac{1+\zeta}{1-\zeta}\in \bbC_+
$$
and the corresponding unitary operator $U:H^2(\bbC_+)\to H^2(\bbD)$, 
$$
(U\f)(\zeta)=\frac{2\sqrt{\pi}}{1-\zeta}\f\left(i\frac{1+\zeta}{1-\zeta}\right). 
$$
We have
$$
U: \frac{\omega(\xi)^j}{\xi+i}\mapsto -i\sqrt{\pi}\psi(\zeta)^j, 
\quad 
\psi(\zeta)=\omega\left(i\frac{1+\zeta}{1-\zeta}\right). 
$$
It follows that $U$ maps the right hand side of \eqref{d1} to the function
$$
-i\sqrt{\pi}\sum_{j\geq0} x(j)\psi(\zeta)^j. 
$$
Since $\abs{\psi(\zeta)}\leq1$, by the Littlewood subordination theorem
\cite[Chap. 1.3]{Shapiro}, we have
$$
\Norm{\sum_{j\geq0}x(j) \psi(\zeta)^j }_{H^2(\bbD)}
\leq
C \Norm{\sum_{j\geq0}x(j) \zeta^j}_{H^2(\bbD)}
=
C \norm{x}_{\ell^2}. 
$$
Putting this together, we obtain the required statement. 
\end{proof}

\begin{proof}[Proof of Theorem~\ref{convolution}]
Let us write
$$
\wc\bvphi(\xi)=\wc\bvphi_1(\xi)\overline{\wc\bvphi_2(-\xi)}, 
\quad
\wc\bvphi_1(\xi)=\wc\bvphi(\xi)(\xi+i), 
\quad
\wc\bvphi_2(\xi)=-\frac{1}{\xi+i}, 
$$
so that $\bvphi=\bvphi_1*\overline{\bvphi_2}$.
By \eqref{convolution1} combined with the condition on the support of $\bvphi$, we have 
$\wc\bvphi_1,\wc\bvphi_2\in H^2(\bbC_+)$ and so 
$\bvphi_1,\bvphi_2\in L^2(0,\infty)$.

For $i=1,2$, let $\Phi_i:\ell^2(\bbZ_+)\to L^2(0,\infty)$ be the map
$$
\Phi_i: x=\{x(j)\}_{j\ge 0}\mapsto \sum_{j\geq0} x(j) T_\nu^j\bvphi_i.
$$
Observe that 
$$
\wc{T_\nu^j\bvphi_i}(\xi)=\wc\nu(\xi)^j \wc\bvphi_i(\xi).
$$
Further, since by hypothesis 
$\nu([0,\infty))\leq1$, we have that the inverse Fourier transform $\wc\nu$ 
is in $H^\infty(\bbC_+)$ with $\norm{\wc\nu}_{H^\infty(\bbC_+)}\leq1$.

Let us first show that $\Phi_2$ is bounded.
By applying the inverse Fourier transform, it 
suffices to check that the map 
$$
x=\{x(j)\}_{j\geq0}\mapsto \left(\sum_{j\geq0} x(j)
\wc\nu(\xi)^j \right)\wc\bvphi_2(\xi)
$$
is bounded from $\ell^2$ to $H^2(\bbC_+)$. 
Recalling the definition of $\wc\bvphi_2$, we see that this immediately follows from Lemma~\ref{lma.d3}. 

To prove that $\Phi_1$ is bounded, we write
$$
\wc\bvphi_1(\xi)=\frac{\bh(\xi)}{\xi+i}, \quad \bh(\xi)=\wc\bvphi(\xi)(\xi+i)^2. 
$$
By \eqref{convolution1}, we have $\bh\in H^\infty(\bbC_+)$, and so the boundedness of $\Phi_1$ again follows by 
an application of Lemma~\ref{lma.d3}.

It remains to check formula \eqref{d2}. 
This is the same argument as the one in the proof of Theorem~\ref{average}. 
Indeed, we have
$$
T_\nu^{j+k}\bvphi
=
T_\nu^{j+k}(\bvphi_1*\overline{\bvphi_2})
=
(T_\nu^ j\bvphi_1)*(\overline{T_\nu^k\bvphi_2}),
$$
and therefore
$$
\alpha(j+k)=\int_0^\infty \ba(t) \int_\bbR (T_\nu^ j\bvphi_1)(t-s)(\overline{T_\nu^k\bvphi_2})(s)\,ds\, dt.
$$
Since $\supp T_\nu^j \bvphi_i\subset [0,\infty)$, by a change of variable
this can be rewritten as
$$
\alpha(j+k)=\int_0^\infty \int_0^\infty \ba(t+s)(T_\nu^ j\bvphi_1)(t)(\overline{T_\nu^k\bvphi_2})(s)\,ds\, dt.
$$
Now we see that 
\begin{align*}
(\bH(\ba)\Phi_1x,\Phi_2x)
& =
\sum_{j,k\geq0}x(j)\overline{x(k)}\int_0^\infty \int_0^\infty 
\ba(t+s)(T_\nu^ j\bvphi_1)(t)(\overline{T_\nu^k\bvphi_2})(s)\,ds\, dt \\
& =
(H(\alpha)x,x).
\end{align*}

\end{proof}

\begin{example}
For $t\ge 0$, let 
$$
\bvphi(t) = -4\pi te^{-2\pi t} \quad \text{and} \quad
\nu(t) = \delta(t) - 4\pi e^{-2\pi t}.
$$
Then 
$$
\wc\bvphi(\xi) = \frac{1}{\pi(\xi+i)^2} \quad \text{and} \quad
\wc\nu(\xi)=\frac{\xi - i}{\xi + i} \in L^\infty(\bbR).
$$
Hence the conclusions of Theorem \ref{convolution} hold. However, we can
say more in this case. We also have that $\bvphi=\bpsi*\bpsi$, with
$\bpsi(t) =-2i\sqrt{\pi}e^{-2\pi t}$, $t\ge 0$, and so we can take
$$
\Phi_1 x=\Phi_2 x= \sum_{j\ge 0} x(j) T_\nu^j\bpsi
$$
in the proof of Theorem \ref{convolution}. It can be shown that
$T_\nu^j\bpsi = \bu_j$, $j\ge 0$, where $\{\bu_j\}_{j\ge 0}$ is the
orthonormal basis given by \eqref{laguerre}. Thus $\Phi_1$ (and hence
$\Phi_2$) is unitary. Consequently, this choice of $\bvphi$ and $\nu$
produces the well-known unitary equivalence between Hankel matrices
and integral Hankel operators.
\end{example}

\end{document}